\documentclass[12pt]{amsart}
\usepackage[colorlinks=true,pagebackref,hyperindex,citecolor=green,linkcolor=blue]{hyperref}
\usepackage{amsfonts,amssymb,amsthm,amsmath,amsxtra,amscd,verbatim,eucal}
\usepackage[all]{xy}
\usepackage{amsmath}
\usepackage{amsfonts}
\usepackage{amssymb}
\usepackage{color}
\usepackage{fullpage}
\usepackage{mathrsfs}
\usepackage{setspace}
\usepackage{moreverb}

\theoremstyle{definition}
\newtheorem{theorem}{Theorem}[section]
\newtheorem{corollary}[theorem]{Corollary}
\newtheorem{lemma}[theorem]{Lemma}
\newtheorem{proposition}[theorem]{Proposition}

\newtheorem{notation}{Setup}[section]

\theoremstyle{definition}
\newtheorem{definition}[theorem]{Definition}

\newtheorem{remark}[theorem]{Remark}
\newtheorem{pregunta}[theorem]{Question}

\newcommand{\m}{\mathfrak{m}}
\newcommand{\fa}{\mathfrak{a}}

\newcommand{\QQ}{\mathbb{Q}}
\newcommand{\cJ}{\mathcal{J}}

\newcommand{\cM}{\mathcal{M}}
\newcommand{\cO}{\mathcal{O}}

\newcommand{\ZZ}{\mathbb{Z}}
\newcommand{\NN}{\mathbb{N}}
\newcommand{\KK}{\mathbb{K}}
\newcommand{\RR}{\mathbb{R}}

\usepackage[textwidth=3.3 cm,textsize=small,shadow,
]{todonotes}

\newcommand{\details}[2][]{} 
 
\newcommand{\Spec}{\operatorname{Spec}}

\newcommand{\Ker}{\operatorname{ker}}

\newcommand{\IM}{\operatorname{Im}}

\newcommand{\Tor}{\operatorname{Tor}}

\renewcommand{\and}{ \quad \text{and} \quad }

\title{Poincar\'e series of multiplier and test ideals }

\author[J. \`Alvarez Montaner ]{Josep \`Alvarez Montaner{$^1$}}
\address{Departament de Matem\`atiques\\
Universitat Polit\`ecnica de Catalunya\\ Av.~Diagonal 647, Barcelona
08028, Spain} \email{josep.alvarez@upc.edu}

\author[]{Luis N\'u\~nez-Betancourt${^2}$}
\address{Centro de Investigaci\'on en Matem\'aticas, Guanajuato, Gto., M\'exico}
\email{luisnub@cimat.mx}

\thanks{{$^1$}Partially supported by grants  MTM2015-69135-P (MINECO/FEDER), 2017SGR-932 (AGAUR) and PID2019-103849GB-I00 (AEI/10.13039/501100011033). }
\thanks{{$^2$}Partially supported by CONACYT Grant 284598 and C\'atedras Marcos Moshinsky.}

\begin{document}
\maketitle

\begin{abstract}
We prove the rationality of the Poincar\'e series of multiplier ideals in any dimension and thus extending the main results  for surfaces of Galindo and Monserrat  and Alberich-Carrami\~nana et al. Our results also hold for Poincar\'e series of  test ideals. In order to do so, we introduce a theory of Hilbert functions indexed over $\RR$ which gives an unified treatment of both cases. 
\end{abstract}

\section{Introduction}

Let $A$ be a  commutative Noetherian ring containing a field $\KK$. 
Assume that $A$ is either local or graded with maximal ideal $\m $ 
and let $\fa$ be  an $\m$-primary ideal.  Depending on the characteristic of the base field we may find two parallel sets of invariants associated to the pair $(A,\fa^c)$ where $c$ is a real parameter. In characteristic zero we have the theory of {\it multiplier ideals} which play a prominent role  in birational geometry and are defined using resolution of singularities  (see \cite{Laz04} for more insight). In positive characteristic we may find the so-called {\it test ideals} which originated from the theory of tight closure \cite{HH, HY} and are defined using the Frobenius endomorphism \cite{BMS, Sch11, Bli}.
Despite its different origins, it is known that under some conditions on $A$, the reduction mod $p$ of a multiplier ideal is the corresponding test ideal \cite{Smi, Har, HY, Tak, MS, dFDTT, CEMS} (see also \cite{SurveyTestIDeals,BFS}). Moreover, both theories share a lot of common properties which we summarize as saying that they form a filtration of  $\m$-primary ideals 
$$
\cJ  : \hskip 5mm A \varsupsetneq \cJ_{\alpha_1}\varsupsetneq \cJ_{\alpha_2}\varsupsetneq\ldots\varsupsetneq  \cJ_{\alpha_i}\varsupsetneq\ldots
$$ and the indices where there is an strict inequality is, under some assumptions on $A$, a discrete set of rational numbers \cite{Laz04, CEMS, BMS, TT, KLZ, BSTZ, Sch211, ST12}. 
The {\it multiplicity} of $c\in \RR_{>0}$ is defined as 
$m(c) = \dim_{\KK} \left(\cJ_{c-\varepsilon} / \cJ_{c}\right),$ for $\varepsilon >0$ small enough \cite{ELSV04}.  In order to gather the information given by these ideals and its multiplicities, we consider the {\it Poincar\'e series} of $\cJ$ $$
P_{\bf \cJ} (T)= \sum_{c\in \RR_{>0}} \dim_{\KK} \left(\cJ_{c-\varepsilon} / \cJ_{c}\right) \hskip 1mm T^{c}.
$$  

The natural question is whether this is a rational function, in the sense that it belongs to the field of fractional functions $\QQ(z)$ where the indeterminate $z$ corresponds to a fractional power $T^{1/e}$ for a suitable $e\in \NN_{>0}$. 

\vskip 2mm

Galindo and Monserrat \cite{GM10} proved that this rationality property holds for  multiplier ideals associated to simple $\m$-primary ideals in a complex smooth surface and provided an explicit formula. These results were extended later on  by Alberich-Carrami\~nana et al. \cite{ACAMDCGA13} (see also \cite{ACAMDCGA18}) to the case of multiplier ideals associated to any $\m$-primary ideal in a complex surface with rational singularities.  The techniques used in both cases rely on the theory of singularities in dimension two and, in particular, the fact that the data coming from the log-resolution of any ideal can be encoded in a combinatorial object such as the \emph{dual graph}. In the case of simple ideals, the divisors corresponding to the star vertices of the graph measure the difference between a multiplier ideal and its preceding. In general one needs the notion of \emph{maximal jumping divisor} \cite{ACAMDCGA13} to account for this difference. The formula obtained for the Poincar\'e series is then described in terms of the \emph{excesses} of these maximal jumping divisors.  During the preparation of this manuscript, we learned that Pande \cite{Pan} has extended these results to the case of smooth varieties in arbitrary dimension.

\vskip 2mm

In this work, we show the rationality of the Poincar\'e  series of multiplier ideals of  $\m$-primary ideals in any  normal variety in arbitrary dimension (see Theorem \ref{thm_mult_rational} and Corollary  \ref{CMcaseMult} for the Cohen-Macaulay case). 
Furthermore, we  also prove the rationality of the Poincar\'e  series  for test ideals of $\m$-primary ideals in $F$-finite rings that are strongly $F$-regular in the punctured spectrum (see Theorem \ref{thm_tau_rational} and Corollary  \ref{CMcaseTest}  for the Cohen-Macaulay case). As a particular case, we obtain the rationality of $P_{\bf \cJ} (T)$ for ideals in normal surfaces in prime characteristic.

\vskip 2mm

Our approach  is completely algebraic, and it  provides an unified proof of the rationality of the Poincar\'e series for both the multiplier and the test ideals in any dimension as long as we have discreteness of the jumping numbers and Skoda's theorem. 
We point out that our main results does not require the rationality of the jumping numbers. Examples of non-rational jumping numbers of multiplier ideals exist by work of Urbinati \cite{Urb12}. The rationality of the Poincar\'e series in this case means  that it belongs to the field of fractional functions 
$\QQ(T^{\alpha_1}, \dots, T^{\alpha_s})$, where $\alpha_1, \dots , \alpha_s \in \RR$ is a finite set of jumping numbers.

\vskip 2mm

To such purpose we develop a theory of Hilbert functions indexed over $\RR$ that should be of independent interest. More precisely, in Section \ref{Sec2} we develop the notion of $\RR$-good $\fa$-filtrations associated to a finitely generated $A$-module which is an extension of the well-known theory of good $\fa$-filtrations. In this general framework we can define the multiplicity of any module in the filtration and the corresponding Poincar\'e series. The main result is Theorem \ref{thm_rational} where we prove the rationality of such a series. In Section \ref{Sec3} we specialize our main result to the case of multiplier ideals and test ideals. We also extend to arbitrary dimension the notion of maximal jumping divisor (see Definition \ref{jumping_divisor})  and give a formula for the multiplicity (see Proposition \ref{multiplicity_formula}).
In Section \ref{Sec4} we provide a different approach to the theory of $\RR$-good $\fa$-filtrations in the case of Cohen-Macaulay rings that gives a simpler formula for the Poincar\'e series (see Theorem \ref{prop1}). By comparing our results with the ones previously obtaining by geometric methods, we yield an algebraic formula for the excess associated to the maximal jumping divisor (see Proposition \ref{PropExcess}).

\vskip 5mm

{ {\bf Acknowledgements:} Part of this work was done during a research stay of the first author at CIMAT, Guanajuato  supported by a Salvador de Maradiaga grant (ref. PRX 19/00405). He wants to thank the people at  CIMAT for the warm welcome.  We are grateful to Swaraj  Pande for sharing a preliminary version of his work. We also acknowledge helpful discussions with V\'ictor Gonz\'alez-Alonso and Mart\'i Lahoz. 
 }

\section{$\RR$-good filtrations} \label{Sec2}

Let $A$ be a  commutative Noetherian ring. 
Assume that $A$ is either local or graded with maximal ideal $\m $ 
and let $\fa$ be  an $\m$-primary ideal.   The theory of {\it good $\fa$-filtrations} gives an approach to the study  of Hilbert functions
that covers most of the classical results in an unified way. We start recalling briefly this notion but we refer to Rossi and Valla's monograph \cite{RV10} and the references therein for more insight.

\vskip 2mm

Let $M$ be a finitely generated $A$-module such that $\lambda\left({M}/{\fa M}\right)< \infty$, where $\lambda(\cdot)$ denotes the length as $A$-module.  A  {\it good $\fa$-filtration} on $M$ is a decreasing filtration 
$$\cM :  \hskip 5mm M=M_0 \supseteq M_1 \supseteq \cdots $$ by $A$-submodules of $M$ such that $M_{j+1}=\fa M_j$ for $j \gg 0$ large enough.
Under these premises we may consider the {\it Hilbert} and  the  {\it Hilbert-Samuel function}  of $M$ with respect to the filtration $\cM$ defined as $$H_{\cM}(j):= \lambda\left({M_j}/{M_{j+1}}\right) \hskip 5mm {\rm and}  \hskip 5mm H^1_{\cM}(j):= \lambda\left({M}/{M_{j}}\right)$$ respectively.
Moreover, we consider the {\it Hilbert} and  the  {\it Hilbert-Samuel series} 
$$ HS_{\cM}(T):=\sum_{j\geq 0} \lambda\left({M_j}/{M_{j+1}}\right)  T^j\hskip 5mm {\rm and}  \hskip 5mm HS^1_{\cM}(T):=\sum_{j\geq 0} \lambda\left({M}/{M_{j}}\right)  T^j.$$ 
Notice that we have $ HS_{\cM}(T)=  (1-T) HS^1_{\cM}(T)$. 
As a consequence of the Hilbert-Serre  Theorem, we can express them as rational functions
$$ HS_{\cM}(T) =(1-T) HS^1_{\cM}(T) = (1-T) \frac {h_{\cM}(T)}{(1-T)^{d+1}}, $$ where $h_{\cM}(T) \in \ZZ[T]$ satisfies $h_{\cM}(1)\neq 0$ and $d$ is the Krull dimension of $M$.  The polynomial $h_{\cM}(T)$ is the {\it $h$-polynomial } of $\cM$.

\vskip 2mm

The aim of this section is to extend the notion of good $\fa$-filtrations by allowing filtrations indexed over $\RR$ and thus mimicking  properties satisfied by filtrations given by multiplier and test ideals.

\begin{definition}\label{def-good}
Let $M$ be a finitely generated $A$-module such that $\lambda\left({M}/{\fa M}\right)< \infty$.
An $\RR$-good  $\fa$-filtration
is a decreasing filtration ${\bf \cM}:=\{M_{\alpha}\}_{\alpha \geq  0}  $ of submodules of $M_0=M$ , indexed by a discrete set of positive real numbers such that 
$M_{\alpha +1}=	\fa M_{\alpha}$ for all $\alpha > j$ with $j \gg 0$ large enough. 
We call it a  $\QQ$-good  $\fa$-filtration when the set of indices is contained in $\QQ$.	
\end{definition}

Indeed, we may think of  ${\bf \cM}$ as a filtration of submodules $M_{c}$  indexed over $\RR$ for which there exist an increasing sequence of real numbers 
$0 < \alpha_1 < \alpha_2 < \ldots$ such that $M_{\alpha_i}=M_c\varsupsetneq M_{\alpha_{i+1}}$ for any $c \in \left[\alpha_i,\alpha_{i+1}\right)$. In particular we have a discrete  filtration of submodules
$$
\cM  :  \hskip 5mm M \varsupsetneq M_{\alpha_1}\varsupsetneq M_{\alpha_2}\varsupsetneq\ldots\varsupsetneq  M_{\alpha_i}\varsupsetneq\ldots
$$ and we say that the $\alpha_i$ are the {\it jumping numbers} of ${\bf \cM}$. A crucial observation is that, once we fix an index $c\in \RR$, 
the filtration 
$$\cM_c :  \hskip 5mm M_c \supseteq M_{c+1} \supseteq M_{c+2} \supseteq \cdots $$
is a good $\fa$-filtration.

\begin{definition}  \label{def-mult}
Let ${\bf \cM}:=\{M_{c}\}_{c \geq  0}  $ be an $\RR$-good  $\fa$-filtration.  We define
the multiplicity of   $c\in \RR_{>0}$ as
$$
m\left(c\right) := \lambda \left({M_{c-\varepsilon}}/{M_c}\right)
$$ for $\varepsilon>0$ small enough. With this definition, it is clear that $c$ is a jumping number if and only if $m\left(c\right)>0$.
\end{definition}

\begin{definition} \label{def-Poincare}
Let ${\bf \cM}:=\{M_{c}\}_{c \geq  0}  $ be an $\RR$-good  $\fa$-filtration.  We define the Poincar\'e series of ${\bf \cM}$ as
$$
P_{\bf \cM} (T)= \sum_{c\in \RR_{>0}} m(c) \hskip 1mm T^{c}.
$$

\end{definition}

The question that we want to address is whether the Poincar\'e series is rational in the sense that it belongs to the field of fractional functions 
$\QQ(T^{\alpha_1}, \dots, T^{\alpha_s})$, where $\alpha_1, \dots , \alpha_s \in \RR$ is a finite set of jumping numbers. In the case of $\QQ$-good  $\fa$-filtrations, the rationality of the Poincar\'e series means that it belongs to the field of fractional functions 
$\QQ(T^{1/e})$ where $e \in \NN_{>0}$ is the least common multiple of the denominators of all the jumping numbers.

\begin{proposition} \label{prop_rational}
Let ${\bf \cM}:=\{M_{c}\}_{c\geq  0}  $ be an $\RR$-good  $\fa$-filtration. Given $c\in \RR_{> 0}$ we have that
$$\sum_{j\geq 0} m(c + j) T^{j}$$ is a rational function in $\QQ(T)$.
\end{proposition}

\begin{proof}
Recall that the Hilbert series $HS^1_{\cM_{c-\varepsilon}}(T)$ and $HS^1_{\cM_c}(T)$  associated to the good $\fa$-filtrations $\cM_{c-\varepsilon}$ and $\cM_{c}$ are rational functions. 
From the short exact sequence 
$$\xymatrix{0\ar[r] & {M_c}/{M_{c+j}}\ar[r]& {M_{c-\varepsilon}}/{M_{c+j}}\ar[r]& {M_{c-\varepsilon}}/{M_{c}}\ar[r]& 0}$$
we get 
$$\sum_{j\geq 0} \lambda \left( {M_{c-\varepsilon}}/{M_{c+j}}\right) T^j = HS^1_{\cM_c}(T) + m(c) \frac{1}{1-T}. $$
Analogously, from the short exact sequence 
$$\xymatrix{0\ar[r] & {M_{c-\varepsilon + j}}/{M_{c+j}}\ar[r]& {M_{c-\varepsilon}}/{M_{c+j}}\ar[r]& {M_{c-\varepsilon}}/{M_{c-\varepsilon + j}}\ar[r]& 0}$$ we get 
\begin{align*}
\sum_{j\geq 0} m(c + j) T^{j} &= \sum_{j\geq 0} \lambda \left( {M_{c-\varepsilon}}/{M_{c+j}}\right) T^j - HS^1_{\cM_{c-\varepsilon}}(T)\\
&= m(c) \frac{1}{1-T} + HS^1_{\cM_c}(T)  - HS^1_{\cM_{c-\varepsilon}}(T) \\
&= \frac{m(c) }{1-T} + \frac {h_{\cM_c}(T) - h_{\cM_{c-\varepsilon}}(T) }{(1-T)^{d+1}}
\end{align*}
and thus it is a rational function. Here,  $h_{\cM_c}(T)$ and $h_{\cM_{c-\varepsilon}}(T)$ are the $h$-polynomials of the good $\fa$-filtrations $\cM_{c-\varepsilon}$ and $\cM_{c}$ respectively.
\end{proof}

\begin{theorem}\label{thm_rational}
Let ${\bf \cM}:=\{M_{c}\}_{c \geq  0}  $ be an $\RR$-good  $\fa$-filtration. Then, the Poincar\'e series $P_{\bf \cM} (T)$ is rational.
Moreover we have $$
P_{\bf \cM} (T)= \sum_{c\in (0,1]} \left( \frac{m(c) }{1-T} + \frac {h_{\cM_c}(T) - h_{\cM_{c-\varepsilon}}(T) }{(1-T)^{d+1}}\right)\hskip 1mm T^{c}, 
$$ where $h_{\cM_c}(T)$ and $h_{\cM_{c-\varepsilon}}(T)$ are the $h$-polynomials of the good $\fa$-filtrations $\cM_{c-\varepsilon}$ and $\cM_{c}$ respectively.
\end{theorem}

\begin{proof}
We have 
$$
P_{\bf \cM} (T)= \sum_{c\in \RR_{>0}} m(c) \hskip 1mm T^{c} =.\sum_{c\in (0,1]} \left(\sum_{j\in\ZZ_{\geq0}} m(c+j) T^j \right)\hskip 1mm T^{c}
$$ and thus the result follows from Proposition \ref{prop_rational}. 
\end{proof}

\section{Poincar\'e series of multiplier and test ideals}  \label{Sec3}
In this section we turn our attention to the case where $A$ contains a field $\KK$ and the $\RR$-good $\fa$-filtration that we  consider is given by a 
filtration of  $\m$-primary ideals 
$$
\cJ  : \hskip 5mm A \varsupsetneq \cJ_{\alpha_1}\varsupsetneq \cJ_{\alpha_2}\varsupsetneq\ldots\varsupsetneq  \cJ_{\alpha_i}\varsupsetneq\ldots
$$
In this setting,  the multiplicity of $c\in \RR_{>0}$ is 
$m(c) = \dim_{\KK} \left(\cJ_{c-\varepsilon} / \cJ_{c}\right),$ for $\varepsilon >0$ small enough, and the Poincar\'e series of $\cJ$ is $$
P_{\bf \cJ} (T)= \sum_{c\in \RR_{>0}} \dim_{\KK} \left(\cJ_{c-\varepsilon} / \cJ_{c}\right) \hskip 1mm T^{c}.
$$

\vskip 2mm

The aim of this section is to specialize the results we obtained in the previous section to the case of multiplier ideals and test ideals.

\subsection{Multiplier ideals}
Let $(A,\m)$ be a normal local ring containing an algebraically closed field $\KK$ of characteristic zero 
and $\fa\subseteq A$ an  ideal. Under these general assumptions  we ensure the existence of canonical divisors $K_X$ on $X=\Spec A$ which are not necessarily $\QQ$-Cartier. Then we may find some effective boundary divisor $\Delta$ such that $K_X + \Delta $ is $\QQ$-Cartier with index $m$ large enough. Now, given a
{\em log-resolution}  $\pi: X' \rightarrow X$ of the triple $(X,\Delta, \fa)$ we pick a canonical divisor $K_{X'}$ in $X'$ such that $\pi_{\ast} K_{X'}=K_X$ and let $F$ be an effective divisor such that $\fa\cdot\cO_{X'} = \cO_{X'}\left(-F\right)$. 

\vskip 2mm

The {\em multiplier ideal} associated to the triple $(X,\Delta, \fa^c)$ for some real number $c \in \RR_{>0}$ is defined as
$$\cJ(X,\Delta,\fa^{c}) = \pi_*\mathcal{O}_{X'}\left(\left\lceil K_{X'} - \frac{1}{m}\pi^{\ast}(m(K_X+\Delta)) - cF \right\rceil\right).$$
This construction allowed de Fernex and Hacon \cite{dFH} to  define the multiplier ideal $\cJ(\fa^{c})$  associated to $\fa$ and $c$ as the unique maximal element of the set of multiplier ideals $\cJ(X,\Delta,\fa^{c})$ where $\Delta$ varies among all the effective divisors such that $K_X + \Delta $ is $\QQ$-Cartier. The key point in their proof is the existence of such a divisor $\Delta$ that realizes the multiplier ideal as $\cJ(\fa^{c})=\cJ(X,\Delta,\fa^{c}) $. 
In this general framework we have that  the {\it local vanishing theorem}
still hold \cite[Theorem 4.1.19]{dFEM}. Namely, for any $c \in \RR_{>0}$ we have $$R^1\pi_*\mathcal{O}_{X'}\left(\left\lceil K_{X'} - \frac{1}{m}\pi^{\ast}(m(K_X+\Delta)) - cF \right\rceil\right)=0.$$

\vskip 2mm

\begin{remark}
If $A$ is $\QQ$-Gorenstein,  the canonical module $K_X$ is $\QQ$-Cartier so no boundary $\Delta$ is required in the definition of multiplier ideal. Namely we have
$$\cJ(\fa^{c}) = \pi_*\mathcal{O}_{X'}\left(\left\lceil K_{X'} - \frac{1}{m}\pi^{\ast}(m K_X) - cF \right\rceil\right).$$

\end{remark}

From its construction we have that the multiplier ideals form a  filtration
$$A\varsupsetneq \cJ(\fa^{\alpha_1})\varsupsetneq \cJ(\fa^{\alpha_2})\varsupsetneq...\varsupsetneq \cJ(\fa^{\alpha_i})\varsupsetneq...$$
and the  $\alpha_i$ where  we have a strict inclusion of ideals are the {\em jumping numbers} of the ideal $\fa$.  

\vskip 2mm
Assume in addition that $\fa$ is an $\m$-primary ideal and thus $F$ is a divisor with exceptional support. Then any multiplier ideal $\cJ(\fa^{c})$ is $\m$-primary as well. To ensure that $\cJ=\{\cJ(\fa^{c})\}_{c\geq 0}$ is an $\RR$-good $\fa$-filtration we notice the following:

\vskip 2mm

\begin{itemize}
 \item[$\cdot$] {\it Discreteness}:  If $\fa$ is $\m$-primary, the number of multiplier ideals in any interval $[c_1, c_2]$ is smaller or equal than $\dim_{\KK} \cJ(\fa^{c_1})/\cJ(\fa^{c_2})$.

\vskip 2mm 

 \item[$\cdot$] {\it Skoda's theorem \cite[Corollary 5.7]{dFH}}: For any  $c > \dim A$ we have $\cJ(\fa^{c}) = \fa \cdot  \cJ(\fa^{c-1}).$

\end{itemize}

\vskip 2mm

There are cases where the jumping numbers are not rational as shown by Urbinati  \cite{Urb12}.
Known cases where the jumping numbers form a discrete set of rational numbers and thus the filtration $\cJ=\{\cJ(\fa^{c})\}_{c\geq 0}$ is a $\QQ$-good $\fa$-filtration are:

\vskip 2mm

\begin{itemize}
\item[$\cdot$] $X$ is $\QQ$-Gorenstein.
\item[$\cdot$] The {\it symbolic Rees algebra} $\mathcal{R}(-(K_X+\Delta)):= \bigoplus_{n\geq 0} \cO_{X}(-n(K_X+\Delta))$ is finitely generated \cite[Remark 2.26]{CEMS}.
\end{itemize}

\vskip 2mm

\begin{theorem}\label{thm_mult_rational}
Let $(A,\m)$ be a normal local ring of dimension $d$ containing an algebraically closed field $\KK$ of characteristic zero, $\fa \subseteq A$ an $\m$-primary ideal and let $\cJ:=\{\cJ(\fa^{c})\}_{c\geq 0}$ be the filtration given by multiplier ideals.  Then, the Poincar\'e series $P_{\bf \cJ} (T)$ is rational. Indeed, we have 
 $$
P_{\bf \cJ} (T)= \sum_{c\in (0,1]} \left( \frac{m(c) }{1-T} + \frac {h_{\cJ(\fa^c)}(T) - h_{\cJ(\fa^{c-\varepsilon})}(T) }{(1-T)^{d+1}}\right)\hskip 1mm T^{c},
$$
where $h_{\cJ(\fa^c)}(T)$ is the $h$-polynomial associated to the multiplier ideal $\cJ(\fa^c)$. 
\end{theorem}

\begin{proof}
The result follows from Theorem \ref{thm_rational}.
\end{proof}

\vskip 2mm

When $A$ is the local ring at a rational singularity of a surface, Alberich-Carrami\~nana et al.  \cite[Theorem 4.1]{ACAMDCGA13} gave a  precise formula for the multiplicity $m(c)$ of any given $c \in \RR_{>0}$, and consequently an explicit description of the Poincar\'e series. We may follow the same approach to get a  partial  extension of their formula.

\begin{definition} \label{jumping_divisor}
Let $(X, \Delta, \fa^c)$ be a triple.
The {\it maximal jumping divisor} associated to $c\in \RR_{>0}$ is 
$$H_c = \left\lceil K_{X'} - \frac{1}{m}\pi^{\ast}(m(K_X+\Delta)) - (c-\varepsilon)F \right\rceil-\left\lceil K_{X'} - \frac{1}{m}\pi^{\ast}(m(K_X+\Delta)) - cF \right\rceil$$ where $\varepsilon$ is  small enough. 
\end{definition}

\begin{remark}
Denote $K_{X'} - \frac{1}{m}\pi^{\ast}(m(K_X+\Delta)) = \sum_i k_i E_i$ and $F=\sum_ i e_i E_i$, where the $E_i$'s are the exceptional components of $\pi$. Then $H_c$ can be defined as the reduced divisor whose
components are the $E_i$ such that $k_i - ce_i \in \ZZ$. In particular we have $H_c=H_{c+1}$ for all $c \in \RR_{>0}$.
\end{remark}

\begin{proposition} \label{multiplicity_formula}
Let $(X, \Delta, \fa^c)$ be a triple. Then, the multiplicity of $c\in \RR_{>0}$ is 
$$m\left(c\right)  = h^0\left(H_c,\cO_{H_c}\left(\left\lceil K_{X'} - \frac{1}{m}\pi^{\ast}(m(K_X+\Delta)) - cF \right\rceil+H_c\right)\right)$$
\end{proposition}

\begin{proof}
To avoid heavy notation, let $K_\pi:= K_{X'} - \frac{1}{m}\pi^{\ast}(m(K_X+\Delta))$.
Consider the short exact sequence
$$0 \longrightarrow \cO_{X'}\left(\left\lceil K_\pi-cF\right\rceil\right) \longrightarrow \cO_{X'}\left(\left\lceil K_\pi-cF\right\rceil + H_c\right) \longrightarrow \cO_{H_c}\left(\left\lceil K_\pi-cF\right\rceil+H_c\right) \longrightarrow 0$$
Pushing it forward to $X$ and applying local vanishing for multiplier ideals we get the
short exact sequence
\begin{multline*}
0 \longrightarrow \pi_*\cO_{X'}\left(\left\lceil K_\pi-cF\right\rceil\right) \longrightarrow \pi_*\cO_{X'}\left(\left\lceil K_\pi-cF\right\rceil+H_c\right) \longrightarrow \\
\longrightarrow H^0\left(H_c,\cO_{H_c}\left(\left\lceil
K_\pi-cF\right\rceil+H_c\right)\right) \otimes \KK_O \longrightarrow
0
\end{multline*}
or equivalently
$$0 \longrightarrow\cJ(\fa^c) \longrightarrow\cJ(\fa^{\left(c-\varepsilon\right)}) \longrightarrow H^0\left(H_c,\cO_{H_c}\left(\left\lceil K_\pi-cF\right\rceil+H_c\right)\right) \otimes \KK_O \longrightarrow 0$$
Therefore the multiplicity of $c$ is just $ m\left(c\right)  = h^0\left(H_c,\cO_{H_c}\left(\left\lceil K_\pi-cF\right\rceil+H_c\right)\right)$.
\end{proof}

\begin{pregunta} \label{question1}
The key ingredient for the explicit formula of the Poincar\'e series of multiplier ideals in dimension $2$ given by  Alberich-Carrami\~nana et al. \cite{ACAMDCGA13} is that the multiplicities satisfy  $m\left(c+k\right) -m\left(c\right)= k \rho_c$, where  $\rho_c:=-F\cdot H_c$ are the {\it excesses} associated to the maximal jumping divisor $H_c$.
Pande \cite{Pan} proved that $m\left(c+j\right)$ is a polynomial function in $j$ of degree less than $d$ in the case of smooth varieties in arbitrary dimension $d$.
These results  motivate the following question regarding multiplicities for $\m$-primary ideals in 
normal rings.
 Is there a polynomial expression in terms of $j$ for
$$
m\left(c+j\right) -m\left(c\right) = h^0\left(H_c,\cO_{H_c}\left(\left\lceil K_\pi - cF \right\rceil+ H_c+ j F\right)\right)   - h^0\left(H_c,\cO_{H_c}\left(\left\lceil K_\pi - cF \right\rceil+H_c\right)\right)?
$$
\end{pregunta}

\subsection{Test ideals}

Let $A$ be a commutative Noetherian ring containing a field $\KK$ of characteristic $p>0$. The theory of test ideals has its origins in the work of Hochster and Huneke on tight closure \cite{HH}. In the case of $A$ being a regular ring, Hara and Yoshida \cite{HY} extended the notion of test ideals to pairs $(A,\fa^c)$ where $\fa\subseteq A$ is an ideal. Their construction has been generalized in subsequent works \cite{BMS, BMS2, TT, BSTZ, Sch11,Bli}
using the theory of {\it Cartier operators}.

\vskip 2mm

Assume that $A$ is $F$-finite. Then, the {\it test ideal}  $\tau(\fa^c)$ associated to $\fa$ and  some real number $c \in \RR_{\geq 0}$ is the smallest nonzero ideal which is compatible with any Cartier operator $\phi \in \bigoplus_{e\geq 0} {\rm Hom}_A(F_{\ast}^eA,A) \cdot F_{\ast}^e \fa^{\lceil cp^e\rceil}$, where $F_{\ast}^e$ is the Frobenius functor.  In this situation we also have  a  filtration
$$A\varsupsetneq \tau(\fa^{\alpha_1})\varsupsetneq \tau(\fa^{\alpha_2})\varsupsetneq...\varsupsetneq \tau(\fa^{\alpha_i})\varsupsetneq...$$
and the  $\alpha_i$ where  we have a strict inclusion of ideals  are called the {\em $F$-jumping numbers} of the ideal $\fa$.

\vskip 2mm

We now give a sufficient condition to have that $\tau(\fa^c)$ is $\m$-primary.

\begin{lemma}
Let $(A, \m)$ be a local $F$-finite Noetherian ring containing a field $\KK$ of characteristic $p>0$ and let $\fa \subseteq A$ be an $\m$-primary ideal. Assume that $A_{\mathfrak{p}}$ is a strongly $F$-regular ring for all prime ideals $\mathfrak{p} \neq \m$. Then,  the test ideals $\tau(\fa^c)$ are $\m$-primary or $A$.
\end{lemma}

\begin{proof}
Since test ideals localize \cite[Proposition 3.2]{Bli}, we have that  $\tau(\fa^c)_{\mathfrak{p} } = \tau(\fa_{\mathfrak{p} }^c) = \tau(A_{\mathfrak{p} }^c) = \tau(A_{\mathfrak{p} }) = A_{\mathfrak{p} } $ for all prime ideals $\mathfrak{p} \neq \m$, because  $A_{\mathfrak{p}}$ is  strongly $F$-regular. Therefore ${\rm rad}(\tau(\fa^c)) \supseteq \m$.
\end{proof}

Under these extra assumptions we have that $\tau=\{\tau(\fa^{c})\}_{c\geq 0}$ is an $\RR$-good $\fa$-filtration:

\vskip 2mm

\begin{itemize}
 \item[$\cdot$] {\it Discreteness}:    If $\fa$ is $\m$-primary and $A$ is strongly $F$-regular in the punctured spectrum,  the number of test ideals in any interval $[c_1, c_2]$ is smaller or equal than $\dim_{\KK} \tau(\fa^{c_1})/\tau(\fa^{c_2})$.

\vskip 2mm 

 \item[$\cdot$] {\it Skoda's theorem \cite{Bli, HT,ST12}}: For any  $c > \dim A$ we have $\tau(\fa^{c}) = \fa \cdot  \tau(\fa^{c-1}).$

\end{itemize}

\vskip 2mm

Known cases where the $F$-jumping numbers form a discrete set of rational numbers and thus the filtration $\tau=\{\tau(\fa^{c})\}_{c\geq 0}$ is a $\QQ$-good $\fa$-filtration are:

\vskip 2mm

\begin{itemize}
\item[$\cdot$] $(A,\m)$ is an $F$-finite,  normal $\QQ$-Gorenstein  local domain \cite{BMS, TT, KLZ, BSTZ, Sch211, ST12}.
\item[$\cdot$] $A$ is an $F$-finite ring which is a direct summand of a regular ring \cite{AHN}.
\end{itemize}

\vskip 2mm

\begin{theorem}\label{thm_tau_rational}
Let $(A,\m)$ be an $F$-finite local ring of dimension $d$ containing a field $\KK$ of characteristic $p>0$ and let $\fa$ be an $\m$-primary ideal. Assume that $A_{\mathfrak{p}}$ is a strongly $F$-regular ring for all prime ideals $\mathfrak{p} \neq \m$.
 Let $\tau =\{\tau(\fa^{c})\}_{c\geq 0}$  be the filtration given by test ideals. 
Then, the Poincar\'e series $P_{\bf \tau} (T)$ is rational. Indeed, we have 
 $$
P_{\bf \tau} (T)= \sum_{c\in (0,1]} \left( \frac{m(c) }{1-T} + \frac {h_{\tau(\fa^c)}(T) - h_{\tau(\fa^{c-\varepsilon})}(T) }{(1-T)^{d+1}}\right)\hskip 1mm T^{c}
$$
where $h_{\tau(\fa^c)}(T)$ is the $h$-polynomial associated to the test ideal $\tau(\fa^c)$. 
\end{theorem}

\begin{proof}
The result follows from Theorem \ref{thm_rational}.
\end{proof}

Motivated by the case of multiplier ideals \cite{GM10,ACAMDCGA13,Pan}, we would like to have a precise description of the multiplicities of $F$-jumping numbers since it would yield a more explicit formula for the Poincar\'e series. More precisely we ask the following

\begin{pregunta} \label{question2}
 Is the multiplicity of test ideals of $\m$-primary ideals in a strongly $F$-regular ring, $m\left(c+j\right)$, 
a polynomial function in $j$ of degree  less than $d$ ?
\end{pregunta}

\section{Poincar\'e series in Cohen-Macaulay rings} \label{Sec4}

Let $(A,\m,\KK)$ be a Cohen-Macaulay local ring of dimension $d$. Let $\fa$ be  an $\m$-primary ideal generated by a regular sequence 
$f_1,\dots, f_d$. Let $\cJ=\{\cJ_{c}\}_{c\geq 0}$ be an $\RR$-good $\fa$-filtration of $\m$-primary ideals satisfying Skoda's theorem so $\cJ_{c}=\fa \cJ_{c-1}$ for all $c>d$. The Poincar\'e series of $\cJ$ is 
$$
P_{\bf \cJ} (T)= \sum_{c\in \RR_{>0}} m(c) \hskip 1mm T^{c} = \sum_{c\in (0,1]} \left(\sum_{j \geq 0 } m(c+j) T^j \right)\hskip 1mm T^{c}
$$ 
and zooming in the summands we have
$$
\sum_{j \geq 0 } m(c+j) T^j = m(c)+m(c+1) T+ \cdots + m(c+d-2) T^{d-2} + T^{d-1}\sum_{j \geq 0 } \lambda(\fa^j \cJ_{c+d-1-\varepsilon}/ \fa^j \cJ_{c+d-1}) T^j $$

The aim of this section is to work towards finding a more explicit formula for the Poincar\'e series in Cohen-Macaulay rings, especially in the case that $\cJ$ is a filtration of multiplier or test ideals where we require that $\KK$ is an infinite field. Namely, 
let $(A, \m)$ be a  local Noetherian ring  containing an infinite field $\KK$ and let $\fa$ be  any $\m$-primary ideal.  Every minimal reduction of $\fa$ can be generated by a superficial sequence of length equal to the analytical spread of $\fa$ \cite[Theorem 8.6.3]{HS}. Since $\fa$ is $\m$-primary, $\ell(\fa)=\dim(A)$. If $A$ is Cohen-Macaulay this superficial sequence is indeed a regular sequence. Therefore we have  $\overline{\fa}=\overline{(f_1,\dots, f_d)}$, where $\overline{(\cdot)}$ denotes the integral closure. Multiplier ideals and test ideal are invariant up to integral closure so we may assume that $\fa$ is generated by a regular sequence.

\vskip 2mm

\begin{notation} \label{notation}
Let $(A,\m,\KK)$ be a Cohen-Macaulay local ring of dimension $d$. Let
 $J\subseteq A$ be an $\m$-primary ideal and $\fa=(f_1, \cdots ,f_d)$ a parameter ideal. 
  Consider a free resolution 
\begin{equation}
\xymatrix{
\cdots \ar[r] &
A^{\beta_2} \ar[r] &
A^{\beta_1} \ar[r] &
A \ar[r] &
A/\fa^j  \ar[r] &
0},
\end{equation}
where $\beta_1={{j+(d-1)}\choose{d-1}}$ is the number of generators of $\fa^j$. After tensoring with $A/J$, we get
\begin{equation}\label{complex}
\xymatrix{
\cdots \ar[r] &
(A/J)^{\beta_2} \ar[r]^{\phi_j^J} &
(A/J)^{\beta_1} \ar[r]^{\varphi_j^J} &
(A/J) \ar[r] &
A/(\fa^j+J)  \ar[r] &
0},
\end{equation}
 The morphisms $\varphi_j^J$ and $\phi_j^J$ plays a role  in what follows. If the ideal $J$ is clear from the context we simply denote $\varphi_j$ and $\phi_j$. Notice also that $\phi_j=0$ for $j=0$.
\end{notation}

\begin{lemma} \label{lemma1}
Let $(A,\m,\KK)$ be a Cohen-Macaulay local ring of dimension $d$. Let
 $J\subseteq A$ be an $\m$-primary ideal and $\fa=(f_1, \cdots ,f_d)$ a parameter ideal. 
 Then, for  every $j\in\ZZ_{>0}$ we have
  $$
 \lambda(J/\fa^j J)=   \lambda(A/\fa^j)-\lambda( \IM \phi_j)+(\beta_1-1)\lambda(A/J)  
 $$
where $\beta_1={{j+(d-1)}\choose{d-1}}$.
\end{lemma}
\begin{proof}
From the short exact sequence,
$ 0\to J\to A\to A/J\to 0,$
we have the induced long exact sequence 
$$
0\to \Tor_1^A(A/\fa^j, A/J)\to J/\fa^j J\to A/\fa^j\to A/(\fa^j+J)\to 0.
$$

Following Notation \ref{notation} we have  $\Tor_1^A(A/\fa^j, A/J)= \ker \varphi_j /\IM \phi_j $ and 
$A/(\fa^j+J)= (A/J) /  \IM \varphi_j$. Then,
\begin{align*}
\lambda(J/\fa^jJ)&=\lambda(A/\fa^j)+\lambda( \Tor_1^A(A/\fa^j,A/J))-\lambda(A/(\fa^j+J) )\\
&=\lambda(A/\fa^j)+[\lambda( \Ker \varphi_j)-\lambda( \IM \phi_j)]-[\lambda(A/J)-\lambda(\IM \varphi_j)]\\
&=\lambda(A/\fa^j)-\lambda( \IM \phi_j )-\lambda(A/J)+[\lambda( \Ker \varphi_j)+\lambda(\IM \varphi_j)]\\
&=\lambda(A/\fa^j)-\lambda( \IM \phi_j)-\lambda(A/J)+\lambda((A/J)^{\beta_1})\\
&=\lambda(A/\fa^j)-\lambda( \IM \phi_j)-\lambda(A/J)+\beta_1\lambda(A/J)\\
&=\lambda(A/\fa^j)-\lambda( \IM \phi_j)+(\beta_1-1)\lambda(A/J)\\
\end{align*}
\end{proof}

\begin{lemma} \label{lemma2}
Let $(A,\m,\KK)$ be a Cohen-Macaulay local ring of dimension $d$. Let
 $J \subseteq K\subseteq A$ be  $\m$-primary ideals and $\fa=(f_1, \cdots ,f_d)$ a parameter ideal. 
 Then, 
$$
\sum_{j \geq 0} \lambda(\fa^jK/\fa^j J)T^j= 
\frac{\lambda(K/J)}{(1-T)^d}+ \sum_{j\geq 1} [\lambda( \IM \phi_j^K)- \lambda( \IM \phi_j^J) ] T^j.
$$ 
\end{lemma}
\begin{proof}
From the short exact sequences 
$$0\to \fa^jK/\fa^j J\to K/\fa^j J\to K/\fa^j K\to 0 \quad
,
\quad
0\to J/\fa^j J\to K/\fa^j J\to K/J\to 0$$
we get $\lambda(\fa^jK/\fa^j J)= \lambda(K/J) + \lambda(J/\fa^j J) - \lambda(K/\fa^j K)$. Thus, applying Lemma \ref{lemma1} to the ideals $J$ and $K$, we get
\begin{align*}
\lambda(\fa^jK/\fa^j J)&=\lambda(K/J)+ [\lambda(A/\fa^j)-\lambda( \IM \phi_j^J)+(\beta_1-1)\lambda(A/J)]\\
& -[\lambda(A/\fa^j)-\lambda( \IM \phi_j^K)+(\beta_1-1)\lambda(A/K)]\\
&=\lambda(K/J)+(\beta_1-1)(\lambda(A/J)-\lambda(A/K))+[\lambda( \IM \phi_j^K)- \lambda( \IM \phi_j^J) ]\\
&= \beta_1\lambda(K/J)+[\lambda( \IM \phi_j^K)- \lambda( \IM \phi_j^J) ],\\
\end{align*} where $\beta_1={{j+(d-1)}\choose{d-1}}$. 
Then the result follows since $\sum_{j\geq 0}{{j+(d-1)}\choose{d-1}} T^j= \frac{1}{(1-T)^d}$.
\end{proof}

In order to get some control on $\lambda( \IM \phi_j)$ we use the following result of Kodiyalam \cite[Theorem 2]{Kod} in the form that we need in the present work.

\begin{proposition}\label{Kod1}
Let $(A,\m,\KK)$ be a  local ring of dimension $d$ and let $\fa, J$ be $\m$-primary ideals. Then, for all $i\geq 0$, the function $\lambda ( \textrm{Tor}_i^A(A/\fa^j, A/J))$ is a polynomial of degree $d-1$ for $j\gg 0$ large enough.
\end{proposition}

Using the additivity of the function $\lambda$ and the fact that Tor modules are the homology modules of the complex (\ref{complex}),  we get

\begin{corollary}\label{Kod2}
Under Setup \ref{notation}, the function $\lambda( \IM \phi_j^J)$ is a polynomial of degree $d-1$ for $j\gg 0$ large enough.
\end{corollary}

The main result of this section is the following

\begin{theorem} \label{prop1}
Let $(A,\m,\KK)$ be a Cohen-Macaulay local ring of dimension $d$. Let $\fa=(f_1, \cdots ,f_d)$ be a parameter ideal and
$\cJ=\{\cJ_{c}\}_{c\geq 0}$  an $\RR$-good $\fa$-filtration of $\m$-primary ideals satisfying  $\cJ_{c}=\fa \cJ_{c-1}$ for all $c>d$.
 Then,  there exists  $\alpha_1, \dots, \alpha_{d} \in \ZZ$ and $p(T) \in \ZZ[T]$  such that 
\begin{align*}
P_{\cJ}(T)= 
\sum_{c\in (0,1]} &  \left(  m(c)+ \cdots + m(c+d-2) T^{d-2} +  \frac{m(c+d-1) T^{d-1}}{(1-T)^d}  \right. \\
& \left.  + T^d \left(\frac{\alpha_d }{(1-T)^d} + \cdots + \frac{\alpha_1 }{(1-T)} + p(T) \right) \right) T^c.\\
\end{align*}
\end{theorem}

\begin{proof}
We have
$$
\sum_{j \geq 0 } m(c+j) T^j = m(c)+m(c+1) T+ \cdots + m(c+d-2) T^{d-2} + T^{d-1}\sum_{j \geq 0 } \lambda(\fa^j \cJ_{c+d-1-\varepsilon}/ \fa^j \cJ_{c+d-1}) T^j $$ so applying Lemma \ref{lemma2} with $K=\cJ_{c+d-1-\varepsilon}$ and $J=\cJ_{c+d-1}$ we get 
\begin{align*}
P_{\cJ}(T)= 
\sum_{c\in (0,1]} &  \left(  m(c)+ \cdots + m(c+d-2) T^{d-2} +  \frac{m(c+d-1) T^{d-1}}{(1-T)^d}  \right. \\
& \left.  + T^{d-1} \sum_{j\geq 1} [\lambda( \IM \phi_j^{\cJ_{c+d-1-\varepsilon}})- \lambda( \IM \phi_j^{\cJ_{c+d-1}}) ] T^j \right) T^c.\\
\end{align*}
Using Corollary \ref{Kod2} we have that for  $j\gg 0$ large enough $\lambda( \IM \phi_j^{\cJ_{c+d-1-\varepsilon}})- \lambda( \IM \phi_j^{\cJ_{c+d-1}}) $ is a polynomial of degree $d-1$ that can be written as
$$
\alpha_{d} {{(j-1)+d-1}\choose{d-1}} + \cdots + \alpha_3 {{(j-1)+2}\choose{2}} + \alpha_2 j +\alpha_1 
$$
Therefore, there exists $k\in \ZZ_{>0}$  such that

\begin{flalign*}
 &T^{d-1} \sum_{j\geq 1} [\lambda( \IM \phi_j^{\cJ_{c+d-1-\varepsilon}})- \lambda( \IM \phi_j^{\cJ_{c+d-1}}) ] T^j   = \\
  &= T^{d} \left( q(T) +\sum_{j\geq k} \left[ \alpha_{d} {{(j-1)+d-1}\choose{d-1}} + \cdots + \alpha_3 {{(j-1)+2}\choose{2}} + \alpha_2 j +\alpha_1\right] T^{j-1} \right)   \\
   & = T^{d} \left( q(T) +   \left(\frac{\alpha_d }{(1-T)^d} - q_d(T) \right)+ \cdots + \left( \frac{\alpha_1 }{(1-T)} - q_1(T) \right) \right)   \\
 \end{flalign*}
where $q(T), q_d(T), \dots, q_1(T)\in \ZZ(T)$ have degree $ \leq k-2$ and the result follows after taking $p(T)= q(T) - q_d(T)- \cdots - q_1(T)$.
\end{proof}

\vskip 2mm

The following result is a direct consequence of Theorem \ref{prop1}. 

\begin{corollary}
Let $(A,\m,\KK)$ be a Cohen-Macaulay local ring of dimension $d$. Let $\fa=(f_1, \cdots ,f_d)$ be a parameter ideal and
$\cJ=\{\cJ_{c}\}_{c\geq 0}$  an $\RR$-good $\fa$-filtration of $\m$-primary ideals satisfying  $\cJ_{c}=\fa \cJ_{c-1}$ for all $c>d$.
Then, the function $m(c+j)$ is a polynomial function on $j$ of degree  less than $d$  for $j\gg 0$ large enough. 
\end{corollary} 

\begin{remark}
In the case of multiplier ideals in a smooth variety,  Pande proved that this  result holds for all $j$  \cite[Theorem 3.2]{Pan}.
\end{remark}

Now we also specialize Theorem \ref{prop1} to the case of multiplier and test  ideals.

\begin{corollary} \label{CMcaseMult}
Suppose  $(A,\m,\KK)$ is a normal Cohen-Macaulay local ring of dimension $d$ over an algebraically closed field of characteristic zero, $\fa \subseteq A$ is any $\m$-primary ideal and  $\cJ:=\{\cJ(\fa^{c})\}_{c\geq 0}$ is the filtration given by multiplier ideals.  
Then, 
\begin{align*}
P_{\cJ}(T)= 
\sum_{c\in (0,1]} &  \left(  m(c)+ \cdots + m(c+d-2) T^{d-2} +  \frac{m(c+d-1) T^{d-1}}{(1-T)^d}  \right. \\
& \left.  + T^d \left(\frac{\alpha_d }{(1-T)^d} + \cdots + \frac{\alpha_1 }{(1-T)} + p(T) \right) \right) T^c.\\
\end{align*}
\end{corollary}
\begin{proof}
For every $\m$-primary ideal $\fa$ there exist a parameter ideal with the same integral closure.
Since the multiplier ideals are  the same for an ideal and its integral closure \cite[Variation 9.6.39]{Laz04} (see also \cite[Corollary 5.7]{dFH}), the result follow from Theorem \ref{prop1}.
\end{proof}

\begin{corollary} \label{CMcaseTest}
Suppose that $(A,\m, \KK)$ is an $F$-finite Cohen-Macaulay local  domain of dimension $d$ over an infinite field  of characteristic $p>0$, $A_{\mathfrak{p}}$ is a strongly $F$-regular ring for all prime ideals $\mathfrak{p} \neq \m$, $\fa$ is any $\m$-primary ideal and
$\tau =\{\tau(\fa^{c})\}_{c\geq 0}$  is the filtration given by test ideals. 
Then,
\begin{align*}
P_{\tau}(T)= 
\sum_{c\in (0,1]} &  \left(  m(c)+ \cdots + m(c+d-2) T^{d-2} +  \frac{m(c+d-1) T^{d-1}}{(1-T)^d}  \right. \\
& \left.  + T^d \left(\frac{\alpha_d }{(1-T)^d} + \cdots + \frac{\alpha_1 }{(1-T)} + p(T) \right) \right) T^c.\\
\end{align*}
\end{corollary}
\begin{proof}
For every $\m$-primary ideal $\fa$ there exist a parameter ideal with the same integral closure, because $\KK$ is infinite.
Since the  test  ideals are the same for an ideal and its integral closure \cite[Proof of Theorem 4.1]{HT} (see also \cite[Lemma 2.27]{BMS}), the result follow from Theorem \ref{prop1}.
\end{proof}

\begin{remark} \label{CM_dim2}
Let $(A,\m, \KK)$ be an $F$-finite normal  local ring of dimension $2$ over an infinite field  of characteristic $p>0$. Then the condition of being strongly $F$-regular in the punctured spectrum and being Cohen-Macaulay is automatically satisfied
\end{remark}

\subsection{The case of multiplier ideals in dimension two revisited}
Let $(A,\m,\KK)$ be a Cohen-Macaulay local ring of dimension $2$. Let $\fa=(f_1, f_2)$ be a parameter ideal and $\cJ=\{\cJ_{c}\}_{c\geq 0 }$ an $\RR$-good $\fa$-filtration of $\m$-primary ideals satisfying  $\cJ_{c}=\fa \cJ_{c-1}$ for all $c>2$. 
Using Theorem \ref{prop1} we get the Poincar\'e series

\begin{equation}\label{eq_dim2}
P_{\cJ}(T)= 
\sum_{c\in (0,1]}   \left(  m(c) +  \frac{m(c+1) T}{(1-T)^2} + T^2 \left(\frac{\alpha_2 }{(1-T)^2}  + \frac{\alpha_1 }{(1-T)} + p(T) \right) \right) T^c.\\
\end{equation}

We  see that, at least for the case of multiplier ideals in a complex surface with a rational singularity, this formula is much simpler. To do so we compare our formula with the one obtained in that case. 

\begin{theorem}[{\cite[Theorem 6.1]{ACAMDCGA13}}]
Let $(A,\m)$ be the local ring of a complex surface  with a rational singularity, $\fa \subseteq A$ an $\m$-primary ideal and let $\cJ:=\{\cJ(\fa^{c})\}_{c\geq 0}$ be the filtration given by multiplier ideals. Then
 $$
P_{\bf \cJ} (T)= \sum_{c\in (0,1]} \left( \frac{m(c) }{1-T} +\frac { \rho_c  T }{(1-T)^{2}}\right)\hskip 1mm T^{c}
$$ 
where  $\rho_c:=-F\cdot H_c$  is  the excess associated to the maximal jumping divisor $H_c$. 
\end{theorem}

If we compare both formulas  we observe
\begin{align*}\label{eq_dim2_2}
P_{\cJ}(T) 
&=\sum_{c\in (0,1]}   \left(  \frac{m(c) + (m(c+1)-2m(c)) T + m(c) T^2}{(1-T)^2} + T^2 \left(\frac{\alpha_2 }{(1-T)^2}  + \frac{\alpha_1 }{(1-T)} + p(T) \right) \right) \hskip 1mm T^{c}.\\
& = \sum_{c\in (0,1]}   \left( \frac{m(c) }{1-T} +\frac { \rho_c  T }{(1-T)^{2}}  +  \frac{T^2}{(1-T)^2} \left( m(c) + \alpha_2 + \alpha_1 (1-T) + p(T) (1-T)^2 \right)    \right) \hskip 1mm T^{c}
\end{align*}
and we conclude that $m(c) = -\alpha_2$, $\alpha_1=0$ and $p(T)=0$. 
If we take a closer look to these conditions we obtain a reformulation of 
\cite[Proposition 4.5]{ACAMDCGA13}
which, in particular, gives an algebraic formula for the excesses.

\begin{proposition}\label{PropExcess}
Let $(A,\m)$ be the local ring of a complex surface  with a rational singularity, $\fa \subseteq A$ an $\m$-primary ideal and let $\cJ:=\{\cJ(\fa^{c})\}_{c\geq 0}$ be the filtration given by multiplier ideals. Then, 
$$ \rho_c  =\frac{1}{j}\Big( \lambda(\Tor_2^A(A/\fa^j, A/\cJ(\fa^{c+1}))) - \lambda(\Tor_2^A(A/\fa^j, A/\cJ(\fa^{c+1-\varepsilon})))\Big)$$
for every $j\geq 1$, where   is  the  excess associated to the maximal jumping divisor $H_c$. 
In particular, 
 $$m(c+j)-m(c)=  \lambda(\Tor_2^A(A/\fa^j, A/\cJ(\fa^{c+1}))) - \lambda(\Tor_2^A(A/\fa^j, A/\cJ(\fa^{c+1-\varepsilon})))$$ 
for every $j\geq 1$. 
\end{proposition}

\begin{proof}
First recall that the morphisms $\phi_j^J$ in Setup \ref{notation} for  an $\m$-primary ideal 
$J\subseteq A$  are 
$$
\xymatrix{
0\ar[r] &
(A/J)^{j} \ar[r]^{\phi_j^J} &
(A/J)^{j+1} \ar[r]^{\varphi_j^J} &
(A/J) \ar[r] &
A/(\fa^j+J)  \ar[r] &
0},
$$ and thus
$\lambda( \IM \phi_j^{J}) = \lambda((A/J)^{j}) - \lambda( \ker \phi_j^{J}) = j  \lambda(A/J) - \lambda(\Tor_2^A(A/\fa^j, A/J)).  $

\vskip 2mm

For simplicity we  denote $\lambda_j^{c+1}$ and $\lambda_j^{c+1-\varepsilon}$ when we refer to $\lambda( \IM \phi_j^{J})$ with $J$ being the multiplier ideals $\cJ(\fa^{c+1})$ and $\cJ(\fa^{c+1-\varepsilon})$ respectively. Then, as in the proof of Theorem \ref{prop1}, we have
$$ \sum_{j\geq 1} [\lambda_j^{c+1-\varepsilon} - \lambda_j^{c+1} ] T^{j-1}    =   q(T) +   \left(\frac{\alpha_2 }{(1-T)^2} - q_2(T) \right)+  \left( \frac{\alpha_1 }{(1-T)} - q_1(T) \right)   $$ where, for some $k\gg 0$, 
\begin{itemize}
\item[] $q(T)= (\lambda_1^{c+1-\varepsilon} - \lambda_1^{c+1}) + (\lambda_2^{c+1-\varepsilon} - \lambda_2^{c+1})  T +\cdots  +(\lambda_{k-1}^{c+1-\varepsilon} - \lambda_{k-1}^{c+1})  T^{k-2}$.
\item[] $q_2(T)= \alpha_2 ( 1+ 2T + \cdots + (k-1) T^{k-2})$.
\item[] $q_1(T)= \alpha_1 ( 1+ T + \cdots + T^{k-2})$.
\end{itemize}

\vskip 2mm

Since  $\alpha_1=0$, $\alpha_2=-m(c)$ and 
$$ 0=p(T)= (\lambda_1^{c+1-\varepsilon} - \lambda_1^{c+1} + m(c)) + (\lambda_2^{c+1-\varepsilon} - \lambda_2^{c+1}+ 2 m(c))  T +\cdots  +(\lambda_{k-1}^{c+1-\varepsilon} - \lambda_{k-1}^{c+1}+(k-1) m(c))  T^{k-2}$$
we get for $j=1,\dots , k-1$

\begin{align*}  
j m(c) & =\lambda_j^{c+1} - \lambda_1^{c+1-\varepsilon} 
 = j \lambda(A/\cJ(\fa^{c+1})) - \lambda(\Tor_2^A(A/\fa^j, A/\cJ(\fa^{c+1}))) \\ & - j \lambda(A/\cJ(\fa^{c+1-\varepsilon})) + \lambda(\Tor_2^A(A/\fa^j, A/\cJ(\fa^{c+1-\varepsilon}))) \\
& = j m(c+1) + \lambda(\Tor_2^A(A/\fa^j, A/\cJ(\fa^{c+1-\varepsilon}))) - \lambda(\Tor_2^A(A/\fa^j, A/\cJ(\fa^{c+1})))
\end{align*} 

\vskip 2mm

\noindent Therefore $$ j \rho_c = \lambda(\Tor_2^A(A/\fa^j, A/\cJ(\fa^{c+1}))) - \lambda(\Tor_2^A(A/\fa^j, A/\cJ(\fa^{c+1-\varepsilon}))) $$
\vskip 2mm 

\noindent
The same formula also holds for $j\geq k$  since we have  $$\lambda_j^{c+1-\varepsilon} - \lambda_j^{c+1} = \alpha_2 j = -m(c) j.$$
\end{proof}

\bibliographystyle{alpha}
\bibliography{References}

\end{document}